\documentclass[12pt]{amsart}

\usepackage{amssymb,latexsym}

\usepackage{color}

\usepackage{enumerate}

\usepackage[T1]{fontenc}

 \usepackage[french,english]{babel}

\makeatletter

\@namedef{subjclassname@2010}{

  \textup{2010} Mathematics Subject Classification}

\makeatother
\newtheorem{thm}{Theorem}[section]
\newtheorem{Con}[thm]{Conjecture}

\newtheorem{lem}[thm]{Lemma}
\newtheorem{pro}[thm]{Proposition}
\theoremstyle{definition}

\newtheorem{rem}[thm]{Remark}

\numberwithin{equation}{section}

\newcommand{\A}{\mathcal{A}}
\newcommand{\X}{\mathbb{X}}

\newcommand{\ex}{\mathbb{E}}
\newcommand{\re}{\textup{Re}}

\newcommand{\pr}{\mathbb{P}}
\newcommand{\ep}{\varepsilon}

\newcommand{\B}{\mathcal{B}}

\newcommand{\F}{F_{\mathbb{X}}}

\newcommand{\M}{\textup{meas}}
\newcommand{\ph}{\varphi}

\newcommand{\newabstract}[1]{%
  \par\bigskip
  \csname otherlanguage*\endcsname{#1}%
  \csname captions#1\endcsname
  \item[\hskip\labelsep\scshape\abstractname.]
}

\begin{document}

\baselineskip=17pt

\title[]{On the distribution of the error terms in the divisor and circle problems}

\author{Youness Lamzouri}

\address{
Universit\'e de Lorraine, CNRS, IECL,  and  Institut Universitaire de France,
F-54000 Nancy, France}

\email{youness.lamzouri@univ-lorraine.fr}

%\date{}

\begin{abstract} We study the distribution functions of several classical error terms in analytic number theory, focusing on the remainder term in the Dirichlet divisor problem $\Delta(x)$.  We first bound the discrepancy between the distribution function of $\Delta(x)$ and that of a corresponding probabilistic random model, improving results of Heath-Brown and Lau. We then determine the shape of its large deviations in a certain uniform range, which we believe to be the limit of our method, given our current knowledge about the linear relations among the $\sqrt{n}$ for square-free positive integers $n$. Finally, we obtain similar results for the error terms in the Gauss circle problem and in the second moment of the Riemann zeta function on the critical line.
\end{abstract}


\thanks{The author is supported by a junior chair of the Institut Universitaire de France.}

\maketitle

%%%%%%%%%%%%%%%%%%%%%%%%%%%%%%%%%%%%%%%%%%%%%%%%%%%%%%%%%%%%%%%%%%%%%%%%%%%%%%%%%%

\section{Introduction}

%\subsection{The  divisor and circle problems}
Let $\Delta(x)$ be the remainder term in the asymptotic formula for the summatory function of the divisor function. More precisely,
$$ \Delta(x):= \sum_{n\leq x} d(n)- x(\log x+2\gamma-1) , 
$$ where $d(n)=\sum_{m\mid n}1$ denotes the divisor function, and $\gamma$ is the Euler-Mascheroni constant. The classical Dirichlet divisor problem is to determine the smallest $\alpha$ for which $\Delta(x)\ll_{\ep}x^{\alpha+\ep}$ holds for any $\ep>0$. Using the hyperbola method, Dirichlet proved that $\alpha\leq 1/2$. This was improved by Vorono\"i to $\alpha\leq 1/3$, and 
the current best result is $\alpha\leq 131/416$ due to Huxley \cite{Hu}. It is widely believed that $\alpha\leq 1/4$, which is optimal in view of Hardy's result from 1916  that\footnote{Recall that for a real valued function $h$ and a positive function $g$, the notation $h(x) = \Omega(g(x))$ means that $\limsup_{x\to \infty} |h(x)|/g(x) > 0$.} 
$$\Delta(x)=\Omega\left((x\log x)^{1/4}(\log_2 x)\right),$$
where here and throughout we let $\log_k$ denote the $k$-th iterate of the natural logarithm function.
Hardy's result was improved by Hafner \cite{Ha} and later by Soundararajan \cite{So} who showed that 
\begin{equation}\label{Eq:SoundDivisor}
\Delta(x)=\Omega\left((x\log x)^{1/4}(\log_2 x)^{\frac{3}{4}(2^{4/3}-1)}(\log_3 x)^{-5/8}\right).  
\end{equation} 

%where 
%\begin{equation}\label{Eq:DefinitionBeta}
%\beta:=2^{4/3}-1.
%\end{equation}
Based on a probabilistic heuristic argument, Soundararajan conjectured that this omega result represents the true maximal order of $\Delta(x)$ up to a factor of $(\log_2x)^{o(1)}.$ 
Our Theorem \ref{Thm:LargeDeviations} below, which concerns the distribution of large values of $\Delta(x)$, gives further evidence in support of the following more precise conjecture: 
\begin{Con}\label{Con:ConjectureDivisor}
There exist positive constants $C_1$, $C_2$ such that  $$C_1 (X\log X)^{1/4}(\log_2 X)^{\frac{3}{4}(2^{4/3}-1)}\leq \max_{x\in [X, 2X]}\Delta(x)\leq C_2(X\log X)^{1/4}(\log_2 X)^{\frac{3}{4}(2^{4/3}-1)},$$
and 
the same is true for $-\Delta(x)$.
\end{Con}
\subsection{The limiting distribution of $\Delta(t)$} In his seminal paper \cite{HB}, Heath-Brown studied the distribution of various well known error terms in analytic number theory, focusing in particular on the Dirichlet divisor problem.  More specifically, 
Heath-Brown showed that $F(t):=t^{-1/4}\Delta(t)$ has a continuous limiting distribution in the sense that, for any interval $I$ we have
\begin{equation}\label{Eq:HBLimitingDP}\lim_{T\to \infty}\frac{1}{T}\M\{t\in [T, 2T] : F(t)\in I\}= \int_{I}f(u) du,
\end{equation}
where $f\in C^{\infty}(\mathbb{R})$ and can be extended to an entire function on $\mathbb{C}$, and $\M$ denotes the Lebesgue measure on $\mathbb{R}$. 
To simplify our notation we define 
$$\pr_{T}(g(t)\in \mathcal{B}):= \frac{1}{T}\M\{t\in [T, 2T] : g(t)\in \mathcal{B}\},$$
for any measurable function $g$ on $[T, 2T]$ and event $\mathcal{B}$. Heath-Brown's result \eqref{Eq:HBLimitingDP} can be stated in probabilistic language, in terms of a certain probabilistic random model. Let $\{\X_n\}_{n\geq 1 \text{ is  square-free}}$ be a sequence of independent random variables, uniformly distributed on $[0,1]$, and consider the following random trigonometric series
$$\F:= \frac{1}{\pi \sqrt{2}}\sum_{n=1}^{\infty} \frac{\mu(n)^2}{n^{3/4}}\sum_{r= 1}^{\infty}\frac{d(nr^2)}{r^{3/2}}\cos\left(2\pi r \X_n-\frac{\pi}{4}\right),$$
which is almost surely convergent by Kolmogorov's two-series Theorem. Then, it follows from the results of Heath-Brown \cite{HB} that $f$ is the probability density function of $F_{\X}$ so that one has 
\begin{equation}\label{Eq:HeathBrownProbability}
\lim_{T\to\infty} \pr_T(F(t)\leq u)= \pr(F_{\X}\leq u),
\end{equation} 
for every fixed $u\in \mathbb{R}.$
The reason behind this probabilistic random model comes from Vorono\"i's summation formula for $\Delta(t)$ (See Equation (12.4.4) of \cite{Ti}, for example) which implies (in its truncated form) that  uniformly for $t\in [T, 2T]$ we have for any fixed $\ep>0$
$$
    F(t)=\frac{1}{\pi \sqrt{2}}\sum_{n\leq T} \frac{d(n)}{n^{3/4}}\cos\left(4\pi \sqrt{nt} -\frac{\pi}{4}\right)+O\left(T^{\ep}\right).
$$
Writing $n=mr^2$ where $m$ is square-free we obtain 
\begin{equation}\label{Eq:Voronoi}
    F(t)=\frac{1}{\pi \sqrt{2}}\sum_{m\leq T} \frac{\mu(m)^2}{m^{3/4}}\sum_{r\leq \sqrt{T/m}}\frac{d(mr^2)}{r^{3/2}}\cos\left(4\pi r\sqrt{mt} -\frac{\pi}{4}\right)+O\left(T^{\ep}\right).
\end{equation}
The proof of \eqref{Eq:HeathBrownProbability}  then relies on Besicovitch's Theorem \cite{Be} which states that   the set $\{\sqrt{n}\}$, where $n$ varies over the square-free positive integers, is linearly independent over $\mathbb{Q}.$

Lau \cite{Lau2} was the first to obtain a quantitative bound on the \emph{discrepancy} between the distribution functions of $F(t)$ and $F_{\X}$, which we define as 
$$ D_{F}(T):= \sup_{u\in \mathbb{R}}\left|\pr_T(F(t)\leq u)-\pr(F_{\X}\leq u)\right|.
$$
Indeed Theorem 2 of \cite{Lau2} states that 
$$D_F(T)\ll \frac{(\log_3 T)^{3/4}}{(\log\log  T)^{1/8}}.$$
Using a different approach, based on ideas from a previous work of the author with Lester and Radziwi\l\l \ \cite{LLR} on the distribution of values of the Riemann zeta function in the critical strip, we obtain a stronger bound on this discrepancy.
\begin{thm}\label{Thm:Discrepancy}
We have 
$$D_F(T)\ll \frac{(\log_3 T)^{9/4}}{(\log\log T)^{3/4}}.$$
\end{thm}
We did not aim at finding the best admissible exponent of $\log_3T$ in our result, so by working harder one can probably improve the factor $(\log_3T)^{9/4}$ in Theorem \ref{Thm:Discrepancy}. However, we believe that the bound $(\log\log  T)^{-3/4+o(1)} $ constitutes the limit of our method. See Remark \ref{Rem:LimitMethod} below for more details.

\subsection{The distribution of large values of $\Delta(t)$ }
Heath-Brown \cite{HB2} showed that the density function in \eqref{Eq:HBLimitingDP} verifies the bound $f(\alpha)\ll \exp(-|\alpha|^{4-\ep})$, for any fixed $\ep$. This implies in particular that for any fixed $\ep$ and large $V$ we have 
$$ \pr(F_{\X}>V), \ \pr(F_{\X}<-V)\ll \exp\left(-V^{4-\ep}\right).$$
Furthermore, it follows from the work of Bleher, Cheng, Dyson and Lebowitz \cite{BCDL}  that\footnote{Although their results are only stated for a variant of the circle problem, the same proofs apply to the Dirichlet divisor problem.}
$$\exp\left(-V^{4+\ep}\right) \ll \pr(F_{\X}>V) \ll \exp\left(-V^{4-\ep}\right),$$
and the same bounds also apply to $\pr(F_{\X}<-V)$. By a more careful study of the Laplace transform of $F_{\X}$, Lau \cite{Lau2} and independently Montgomery (unpublished), refined these bounds and proved that one has 
\begin{equation}\label{Eq:LAUBounds}
    \exp\left(-\beta_1 V^4(\log V)^{-3(2^{4/3}-1)}\right)\leq \pr(F_{\X}>V)\leq \exp\left(-\beta_2 V^4(\log V)^{-3(2^{4/3}-1)}\right),
\end{equation}
and the same bounds also hold for $\pr(F_{\X}<-V)$, where $\beta_1, \beta_2$ are positive constants. 
One notices that all these results only concern the probabilistic random model $F_{\X}$. Hence, 
a natural question to ask is whether one can obtain similar results for the distribution function $\pr_T(F(T)>V)$ in some uniform range for $V$ in terms of $T$. We answer this in the affirmative in the following theorem. \begin{thm}\label{Thm:LargeDeviations} 
There exists positive constants $b_1, b_2, b_3, b_4$ such that for all real numbers $V$ in the range $b_1\leq V\leq b_2(\log_2 T)^{1/4}(\log_3 T)^{2^{4/3}-7/4}$ we have 
$$\exp\left(-b_3 V^4(\log V)^{-3(2^{4/3}-1)}\right)\leq \pr_T(F(t)>V)\leq \exp\left(-b_4 V^4(\log V)^{-3(2^{4/3}-1)}\right).$$
    Moreover, the same bounds hold for $\pr_T(F(t)<-V)$, in the same range of $V.$
\end{thm}
We observe that if the bounds of Theorem \ref{Thm:LargeDeviations} were to persist to the end of the viable range of $V$, then Conjecture \ref{Con:ConjectureDivisor} would follow.
  
\begin{rem}\label{Rem:LimitMethod} From Vorono\"i's approximate summation formula \eqref{Eq:Voronoi} one can draw some comparison between $F(t)$ and $\log\zeta(3/4+it)\approx\sum_{p\leq t} \frac{e^{-it\log p}}{p^{3/4}}$, where $\zeta(s)$ is the Riemann zeta function\footnote{Unconditionally, the approximation  $\log\zeta(3/4+it)\approx\sum_{p\leq t} \frac{e^{-it\log p}}{p^{3/4}}$ is only valid for $t\in[T, 2T]$ outside a set of measure $O(T^{1-\ep})$, by the classical zero density estimates.}.
The key to obtaining quantitative bounds  on the discrepancy in both cases  is an effective version of the linear independence of the arguments. In the case of $\log\zeta(3/4+it)$, if $\ep_j=\pm1$ and $p_1,\dots, p_m\leq M$ are primes such that $\sum_{j=1}^m \ep_j \log p_j\neq 0$ then one has $|\sum_{j=1}^m \ep_j \log p_j|\gg M^{-m}$, while in the case of $\{\sqrt{n}\}_{n\textup{ is square-free}}$ one has an exponentially weaker bound. Indeed Lemma 3.5 of \cite{HuRu} (which corresponds to Lemma \ref{Lem:HughesRudnick} below) 
gives the bound $|\sum_{j=1}^m\ep_j\sqrt{n_j}|> (mM)^{-2^m}$ if this sum is non-zero and the $n_j$'s are below $M$.  Furthermore, improving this bound to something like $(mM)^{-Q(m)}$, where $Q(m)$ is a polynomial, is a very difficult problem which is related to some major open problems in computational geometry %such as the Euclidean traveling salesman problem 
(see for example \cite{CMSC} and \cite{EHS}). In particular, any significant improvement in the bound $(mM)^{-2^m}$  will immediately yield an improved bound for the discrepancy in Theorem \ref{Thm:Discrepancy} (and in Theorems \ref{Thm:DiscrepancyCircle} and \ref{Thm:DiscrepancyE(T)} below), as well as in the range of uniformity in Theorem \ref{Thm:LargeDeviations} (and in Theorems \ref{Thm:LargeDeviationsCircle} and \ref{Thm:LargeDeviationsE(T)} below).  Now, given that the best current bound on the discrepancy for the distribution of $\log\zeta(3/4+it)$ is $(\log T)^{-3/4},$ which was obtained by Lamzouri, Lester and Radziwi\l\l \ \cite{LLR},  we believe that the bound $(\log\log T)^{-3/4+o(1)}$ for $D_F(T)$ (and $(\log\log T)^{1/4+o(1)}$ for the range of uniformity in Theorem \ref{Thm:LargeDeviations}) constitute the limit of current methods. 
\end{rem}
%PERHAPS THE UNIFORMITY REmark here.

\subsection{Analogous results for the remainder term in the circle problem}
The Gauss's circle problem asks to determine the smallest $\alpha$ for which
$P(x)\ll_{\ep}x^{\alpha+\ep}$ holds for any $\ep>0$, where $P(x)$ is the remainder term in the asymptotic formula for  the number of integer lattice points in a disc of radius $\sqrt{x}$. More precisely 
$$ P(x):=\sum_{n\leq x}r(n)-\pi x, $$
where $r(n)$ is the number of ways of writing $n$ as a sum of two squares. It is widely believed that $\alpha\leq 1/4$, and the best upper bound is Huxley's result \cite{Hu} that $P(x)\ll x^{131/416+\ep}$ for any fixed $\ep>0$. On the other hand the best omega result is due to Soundararajan \cite{So} and states that
\begin{equation}\label{Eq:SoundCircle}
    P(x)=\Omega\left( (x\log x)^{1/4}(\log_2 x)^{\frac{3}{4}(2^{1/3}-1)}(\log_3 x)^{-5/8}\right).
\end{equation}
Using the same method to prove \eqref{Eq:HBLimitingDP}, Heath-Brown showed that $t^{-1/4}P(t)$ has a limiting distribution, which can be stated in probabilistic language as 
\begin{equation}\label{Eq:HBLimitingCircle}
    \lim_{T\to\infty}  \pr_T(t^{-1/4}P(t)\leq u)= \pr(P_{\X}\leq u),
\end{equation} 
where  $P_{\X}$ is the following random trigonometric series 
$$P_{\X}:=- \frac{1}{\pi }\sum_{n=1}^{\infty} \frac{\mu(n)^2}{n^{3/4}}\sum_{q= 1}^{\infty}\frac{r(nq^2)}{q^{3/2}}\cos\left(2\pi q \X_n+\frac{\pi}{4}\right),$$
and $\{\X_n\}_{n \textup{ square-free}}$ is a sequence of i.i.d. random variables uniformly distributed on $[0,1]$. Using our approach we obtain the following results in this case.%, which improve those of Heath-Brown \cite{HB} and Lau \cite{Lau1}. 
\begin{thm}\label{Thm:DiscrepancyCircle}
We have 
$$\sup_{u\in \mathbb{R}}\left|\pr_T(t^{-1/4}P(t)\leq u)-\pr(P_{\X}\leq u)\right|\ll \frac{(\log_3 T)^{5/4}}{(\log\log T)^{3/4}}.$$
\end{thm}
\begin{thm}\label{Thm:LargeDeviationsCircle} 
There exists positive constants $b_5, b_6, b_7, b_8$ such that for all real numbers $V$ in the range $b_5\leq V\leq b_6(\log_2 T)^{1/4}(\log_3 T)^{2^{4/3}-17/12}$ we have 
$$\exp\left(-b_7 V^4(\log V)^{-3(2^{1/3}-1)}\right)\leq \pr_T(t^{-1/4}P(t)>V)\leq \exp\left(-b_8 V^4(\log V)^{-3(2^{1/3}-1)}\right).$$
    Moreover, the same bounds hold for $\pr_T(t^{-1/4}P(t)<-V)$, in the same range of $V.$
\end{thm}
Soundararajan \cite{So} conjectured that his  omega result \eqref{Eq:SoundCircle} represents the true maximal order of $P(x)$ up to $(\log_2x)^{o(1)}.$ Our Theorem \ref{Thm:LargeDeviationsCircle} gives evidence in support of the following more precise conjecture: 
\begin{Con}\label{Con:ConjectureCircle}
There exist positive constants $C_1, C_2$ such that $$C_1(X\log X)^{1/4}(\log_2 X)^{\frac{3}{4}(2^{1/3}-1)}\leq \max_{x\in [X, 2X]}P(x)\leq C_2(X\log X)^{1/4}(\log_2 X)^{\frac{3}{4}(2^{1/3}-1)} ,$$
and 
the same is true for $-P(x)$.
\end{Con}
\subsection{The error term in the second moment of the Riemann zeta function.}
Let 
$$E(T):=\int_{0}^T\left|\zeta\left(\frac12+it\right)\right|^2dt-T\log\left(\frac{T}{2\pi}\right)-(2\gamma-1)T.$$
 It is widely believed that $E(T)\ll T^{1/4+\ep}$, which is best possible in view of Lau and Tsang result \cite{LaTs}, who extended Soundararajan's method to this setting\footnote{Soundararajan's method cannot be directly applied to this case because of the oscillating factor $(-1)^{nr^2}$ in \eqref{Eq:RandomModelE(T)}.} and showed that the omega result \eqref{Eq:SoundDivisor} also holds for $E(T)$. For a complete history as well as recent developments concerning $E(T)$, the reader may consult the expository paper of Tsang \cite{Ts}. Using Atkinson's formula \cite{At} for $E(T)$, Heath-Brown \cite{HB} proved that $t^{-1/4}E(t)$ has a limiting distribution. In this case the corresponding random trigonometric series is 
\begin{equation}\label{Eq:RandomModelE(T)} E_{\X}:=\left(\frac{2}{\pi}\right)^{1/4}\sum_{n=1}^{\infty} \frac{\mu(n)^2}{n^{3/4}}\sum_{r= 1}^{\infty}(-1)^{nr^2}\frac{d(nr^2)}{r^{3/2}}\cos\left(2\pi r \X_n-\frac{\pi}{4}\right),
\end{equation}
where $\{\X_n\}_{n \textup{ square-free}}$ is a sequence of i.i.d. random variables uniformly distributed on $[0,1]$.
%Because of the osciallating factor, Soundararajans' method for 
Our method applies to this setting as well and yield the following results.
\begin{thm}\label{Thm:DiscrepancyE(T)}
We have 
$$\sup_{u\in \mathbb{R}}\left|\pr_T(t^{-1/4}E(t)\leq u)-\pr(E_{\X}\leq u)\right|\ll \frac{(\log_3 T)^{9/4}}{(\log\log  T)^{3/4}}.$$
\end{thm}
\begin{thm}\label{Thm:LargeDeviationsE(T)} 
There exists positive constants $b_9, b_{10}, b_{11}, b_{12}$ such that for all real numbers $V$ in the range $b_{9}\leq V\leq b_{10}(\log_2 T)^{1/4}(\log_3 T)^{2^{4/3}-7/4}$ we have 
$$\exp\left(-b_{11} V^4(\log V)^{-3(2^{4/3}-1)}\right)\leq \pr_T(t^{-1/4}E(t)>V)\leq \exp\left(-b_{12} V^4(\log V)^{-3(2^{4/3}-1)}\right).$$
Moreover, the same bounds hold for $\pr_T(t^{-1/4}E(t)<-V)$, in the same range of $V.$
\end{thm}

\subsection{Final remarks.}
Based on Theorem \ref{Thm:LargeDeviationsE(T)} we believe that  Conjecture \ref{Con:ConjectureDivisor} holds verbatim for $E(T)$. Finally, we note that our results should extend without much effort to  the error term in the  Piltz divisor problem, commonly denoted  $\Delta_3(t)$. However, to keep the paper short and self-contained we decided to not pursue this here.

%%%%%%%%%%%%%%%%%%%%%%%%%%%%%%%%%%%%%%%%%%%%%%%%%%%%%%%%%%%%%%%%%%%%%%%%%%%%%%%%%%%%%%%%%%%%%%%
\section{Properties of the random trigonometric series}
 We first record the following bounds for the Laplace transforms of the random trigonometric series $F_{\X}$, $P_{\X}$ and $E_{\X}$ proved by Lau \cite{Lau1}. 
\begin{lem}\label{Lem:Lau}
\begin{itemize}
    \item[1.] There exist positive constants $c_1, c_2$ such that for every real number $\lambda$ with $|\lambda|>2$  we have 
$$
\exp\left(c_1 |\lambda|^{4/3}(\log |\lambda|)^{2^{4/3}-1}\right)\leq \ex\left(e^{\lambda F_{\X}}\right)\leq \exp\left(c_2 |\lambda|^{4/3}(\log |\lambda|)^{2^{4/3}-1}\right).
$$
Moreover, the same estimates hold for $\ex\left(e^{\lambda E_{\X}}\right)$. 
\item[2.] There exist positive constants $c_3, c_4$ such that for every real number $\lambda$ with $|\lambda|>2$  we have 
$$
\exp\left(c_3 |\lambda|^{4/3}(\log |\lambda|)^{2^{1/3}-1}\right)\leq \ex\left(e^{\lambda P_{\X}}\right)\leq \exp\left(c_4 |\lambda|^{4/3}(\log |\lambda|)^{2^{1/3}-1}\right).
$$ 
\end{itemize}
%where $\beta$ is defined \eqref{Eq:DefinitionBeta}.
\end{lem}
\begin{proof}
The bounds for $\ex\left(e^{\lambda E_{\X}}\right)$ follow from the proof of Theorem 1 of \cite{Lau1}, while those for $\ex\left(e^{\lambda E_{\X}}\right)$ and $\ex\left(e^{\lambda P_{\X}}\right)$ follow from the proof of Theorem 2 of \cite{Lau1}.
\end{proof}

Next, we establish upper bounds for all integral moments of truncations of the random models $F_{\X}$, $P_{\X}$ and $E_{\X}$. We do this in a slightly general manner to cover all cases at once. 
\begin{pro}\label{Pro:BoundMomentsRandom}
Let $\beta_0$ be a real constant. Let $\{a_n\}_n$ be a sequence of real numbers such that $|a_n|\leq 1$. There exist positive constants $c_5,c_6$ such that for any positive integers   $1\leq N< M$, $L\geq 2$ and $k\geq 1$ we have  
\begin{equation}\label{Eq:BoundMomentsTail}
\begin{aligned}
&\ex\left(\left|\sum_{N\leq n< M} \frac{\mu(n)^2}{n^{3/4}}\sum_{q\leq L}a_{nq^2}\frac{d(nq^2)}{q^{3/2}}\cos\left(2\pi  q \X_n+ \beta_0\right)\right|^{k}\right)\\
& \leq \min\left\{\left(c_5k^{1/4}(\log 2k)^{5/4}\right)^{k}, \left(c_6k\frac{(\log N)^3}{\sqrt{N}}\right)^{k/2}\right\},
\end{aligned}
\end{equation}
and 
\begin{equation}\label{Eq:BoundMomentsTailCircle}
\begin{aligned}
&\ex\left(\left|\sum_{N\leq n< M} \frac{\mu(n)^2}{n^{3/4}}\sum_{q\leq L}a_{nq^2}\frac{r(nq^2)}{q^{3/2}}\cos\left(2\pi  q \X_n+ \beta_0\right)\right|^{k}\right)\\
& \leq \min\left\{\left(c_5k^{1/4}(\log 2k)^{1/4}\right)^{k}, \left(c_6k\frac{\log N}{\sqrt{N}}\right)^{k/2}\right\}.
\end{aligned}
\end{equation}  
\end{pro}

\begin{proof}
Since $\ex(|Z|^k)\leq \ex(|Z|^{2k})^{1/2}$ by the Cauchy-Schwarz inequality (where $Z$ is  any random variable), it suffices to prove the same estimates with $2k$ instead of $k$.  We start by establishing \eqref{Eq:BoundMomentsTail} first. We observe that 
\begin{align*}
   &\ex\left(\left|\sum_{N\leq n < M} \frac{\mu(n)^2}{n^{3/4}}\sum_{q\leq L}a_{nq^2}\frac{d(nq^2)}{q^{3/2}}\cos\left(2\pi  q \X_n+\beta_0\right)\right|^{2k}\right)\\
   & = \ex\left(\left|\re\left(e^{i\beta_0}\sum_{N\leq n<M} \frac{\mu(n)^2}{n^{3/4}}\sum_{q\leq L}a_{nq^2}\frac{d(nq^2)}{q^{3/2}}e^{2\pi i  q \X_n}\right)\right|^{2k}\right)\\& \leq \ex\left(\left|\sum_{N\leq n<M} \frac{\mu(n)^2}{n^{3/4}}\sum_{q\leq L}a_{nq^2}\frac{d(nq^2)}{q^{3/2}}e^{2\pi i  q \X_n}\right|^{2k}\right).
\end{align*}
We first prove an upper bound on this moment of the form $\left(ck (\log N)^3/\sqrt{N}\right)^{k}$, for some positive constant $c$.   Expanding this  moment we find that it equals
 \begin{align*}
   &\sum_{N\leq n_1, n_2, \dots, n_{2k}< M}\frac{\mu(n_1)^2\cdots \mu(n_{2k})^2}{(n_1\cdots n_{2k})^{3/4}}\sum_{q_1, q_2, \dots, q_{2k}\leq L} \prod_{j=1}^{2k}
a_{n_jq_j^2}\frac{d(n_jq_j^2)}{q_j^{3/2}} \\
   & \quad \quad \quad \times \ex\left(e^{2\pi i\big((q_1 \X_{n_1}+\cdots +q_k\X_{n_k})-(q_{k+1} \X_{n_{k+1}}+\cdots +q_{2k}\X_{n_{2k}})\big)}\right).
\end{align*} 
Since the $\{\X_n\}$ are independent the inner expectation equals $0$ unless $\{n_1, \dots, n_k\}=\{n_{k+1}, \dots, n_{2k}\}$ and $q_1 \X_{n_1}+\cdots +q_k\X_{n_k}=q_{k+1} \X_{n_{k+1}}+\cdots +q_{2k}\X_{n_{2k}}$, in which case it equals $1$. Moreover,  this condition implies that $q_{i}=q_{j}$ if $n_i=n_j$ for $1\leq i\leq k$ and $k+1\leq j\leq 2k$. Hence we deduce that
\begin{equation}\label{Eq:FirstBoundMoments}
\begin{aligned}
\ex\left(\left|\sum_{N\leq n<M} \frac{\mu(n)^2}{n^{3/4}}\sum_{q\leq L}a_{nq^2}\frac{d(nq^2)}{q^{3/2}}e^{2\pi i  q \X_n}\right|^{2k}\right)
&\leq  k! \left(\sum_{N\leq n<M}\frac{1}{n^{3/2}}\sum_{q\leq L} \frac{d(nq^2)^2}{q^3}\right)^k\\
&\leq  \left(k\sum_{n\geq N}\frac{d(n)^2}{n^{3/2}}\sum_{q\geq 1} \frac{d(q)^4}{q^3}\right)^k\\
&\leq \left(c_7 k \sum_{n\geq N}\frac{d(n)^2}{n^{3/2}}\right)^k,
\end{aligned}
\end{equation}
for some positive constant $c_7$, since $d(m\ell)\leq d(m)d(\ell)$ for all positive integers $m, \ell$, and using that the inner series over $q$ is convergent. The desired bound follows from the estimate $\sum_{n\geq N}\frac{d(n)^2}{n^{3/2}}\ll (\log N)^3/\sqrt{N}$, which is easily deduced from partial summation and the classical bound $\sum_{n\leq x} d(n)^2\ll x(\log x)^3.$

We now define $J:=k\log(2k)$. To complete the proof of \eqref{Eq:BoundMomentsTail}, we might assume that $N< J < M$. Indeed, if $N\geq J$ then the result follows from the previous bound, while if $M\leq J$, then we have trivially
\begin{equation}\label{Eq:Trivial}
\left|\sum_{N\leq n<M} \frac{\mu(n)^2}{n^{3/4}}\sum_{q\leq L}a_{nq^2}\frac{d(nq^2)}{q^{3/2}}e^{2\pi i  q \X_n}\right| \leq \sum_{n\leq J} \frac{d(n)}{n^{3/4}} \sum_{q=1}^{\infty}\frac{d(q)^2}{q^{3/2}}\ll k^{1/4}(\log 2k)^{5/4},
\end{equation}
which follows from the inequality $d(nq^2)\leq d(n)d(q)^2$ together with the classical estimate $\sum_{n\leq J} d(n)/n^{3/4}\ll J^{1/4}\log J$, and the fact that the inner series over $q$ is convergent. We now use the elementary inequality $|a+b|^{2\ell}\leq 2^{2\ell} (|a|^{2\ell}+|b|^{2\ell}),$ valid for all real numbers $a, b$ and positive integers $\ell$, to  obtain in this case
\begin{equation}\label{Eq:SecondBoundMomentRandom}
\begin{aligned}
& \ex\left(\left|\sum_{N\leq n<M} \frac{\mu(n)^2}{n^{3/4}}\sum_{q\leq L}a_{nq^2}\frac{d(nq^2)}{q^{3/2}}e^{2\pi i  q \X_n}\right|^{2k}\right) \\
& \leq  2^{2k}\ex\left(\left|\sum_{N\leq n<J} \frac{\mu(n)^2}{n^{3/4}}\sum_{q\leq L}a_{nq^2}\frac{d(nq^2)}{q^{3/2}}e^{2\pi i  q \X_n}\right|^{2k}\right)\\
& \quad \quad + 2^{2k}\ex\left(\left|\sum_{J\leq n<M} \frac{\mu(n)^2}{n^{3/4}}\sum_{q\leq L}a_{nq^2}\frac{d(nq^2)}{q^{3/2}}e^{2\pi i  q \X_n}\right|^{2k}\right)\\
& \leq \left(c_8k^{1/4}(\log 2k)^{5/4}\right)^{2k} +\left(c_9 k \frac{(\log J)^3}{\sqrt{J}}\right)^k,
\end{aligned}
\end{equation}
for some positive constants $c_8, c_9$ by \eqref{Eq:FirstBoundMoments} and \eqref{Eq:Trivial}. This completes the proof of \eqref{Eq:BoundMomentsTail}.

 We now establish \eqref{Eq:BoundMomentsTailCircle}. Since the proof is similar to that of \eqref{Eq:BoundMomentsTail} we only indicate where the main changes occur. We start by recording some important and classical facts about the function $r(n)$. First we have 
 $r(n)\ll_{\ep}n^{\ep}
 $ for any fixed $\ep>0$. Moreover, we know that the function $\delta(n)=r(n)/4$ is multiplicative and $\delta(p^{a+b})\leq \delta(p^a)\delta(p^{b})$. Therefore, we deduce that \begin{equation}\label{Eq:SubMulti_r(n)}r(nq^2)\ll_{\ep}q^{\ep}r(n).
 \end{equation}
 In addition we shall require the following asymptotic formula due to Ramanujan \cite{Ra}
 $$ 
 \sum_{n\leq x} r(n)^2=4x\log x+O(x).
 $$
 By partial summation, this gives 
\begin{equation}\label{Eq:BoundRamanujan}
 \sum_{n\geq N}\frac{r(n)^2}{n^{3/2}}\ll \frac{\log N}{\sqrt{N}}.
 \end{equation}
 Using this bound together with \eqref{Eq:SubMulti_r(n)} with $\ep=1/2$  and following the proof of \eqref{Eq:FirstBoundMoments} we find that 
\begin{equation}\label{Eq:FirstBoundMomentsCircle}
\begin{aligned}
\ex\left(\left|\sum_{N\leq n<M} \frac{\mu(n)^2}{n^{3/4}}\sum_{q\leq L}a_{nq^2}\frac{r(nq^2)}{q^{3/2}}e^{2\pi i  q \X_n}\right|^{2k}\right)
&\leq  k! \left(\sum_{N\leq n<M}\frac{1}{n^{3/2}}\sum_{q\leq L} \frac{r(nq^2)^2}{q^3}\right)^k\\
&\leq  \left(c_{10}k\sum_{n\geq N}\frac{r(n)^2}{n^{3/2}}\right)^k
\leq \left(c_{11} k \frac{\log N}{\sqrt{N}}\right)^k,
\end{aligned}
\end{equation}
for some positive constants $c_{10}, c_{11}.$ Furthermore, similarly to 
\eqref{Eq:SecondBoundMomentRandom} we derive
\begin{align*}
\ex\left(\left|\sum_{N\leq n<M} \frac{\mu(n)^2}{n^{3/4}}\sum_{q\leq L}a_{nq^2}\frac{r(nq^2)}{q^{3/2}}e^{2\pi i  q \X_n}\right|^{2k}\right)
& \leq \left(c_{12}J^{1/4}\right)^{2k} +\left(c_{11} k \frac{\log J}{\sqrt{J}}\right)^k
\\
&
\ll \left(c_{13}k^{1/4}(\log 2k)^{1/4}\right)^{2k},
\end{align*}
for some positive constants $c_{12}, c_{13}$. This completes the proof.
\end{proof}

%%%%%%%%%%%%%%%%%%%%%%%%%%%%%%%%%%%%%%%%%%%%%%%%%%%%%%%%%%%%%%
\section{Computing moments of truncations of $F(t)$, $t^{-1/4}P(t)$ and $t^{-1/4}E(T)$}

In this section we prove the following general result which implies in particular that the moments (in a certain uniform range) of truncations of $F(t)$ are very close to those of a corresponding probabilistic random model. This will also apply to the case of the circle problem and that of the error term in the second moment of $\zeta(1/2+it)$. 
\begin{pro}\label{Pro:MomentsCalculations}
Let $a_m$ be real numbers such that $a_m\ll 1$. Let $\alpha_0\neq0$ and $\beta_0$ be fixed  real numbers. Let $T$ be large,  $0\leq h \leq (\log\log T)/4$ be an integer and $M\leq T^{1/(2^{h}+4h)}/h^2 $ be a real number. Write $m=nr^2$ where $n$ is square-free and let $\{\X_n\}_{n \ \textup{square-free}}$ be a sequence of independent random variables uniformly distributed on $[0,1]$. Then we have 
\begin{align*}
\frac1T\int_{T}^{2T}\left(\sum_{m\leq M} a_m \cos\left(2\pi \alpha_0\sqrt{mt}+\beta_0\right)\right)^h dt &= \ex\left(\Big(\sum_{nr^2\leq M} a_{nr^2} \cos\big(2\pi r\X_n+\beta_0\big)\Big)^h\right)\\
& \quad + O\left(T^{-1/4}\right),
\end{align*}
where the implied constant is absolute.  
\end{pro}
To prove this result we need the following  lemma from \cite{HuRu}. A weaker version of this result appeared in \cite{HB} (see Lemma 5 there).
\begin{lem}[Lemma 3.5 of \cite{HuRu}]\label{Lem:HughesRudnick} Let $m\geq 1$ and $n_1, \dots, n_m\leq M$ be positive integers. Let $\varepsilon_j=\pm 1$ be such that 
$\sum_{j=1}^m\varepsilon_j\sqrt{n_j}\neq 0$. Then we have 
$$\left|\sum_{j=1}^m\varepsilon_j\sqrt{n_j}\right|\geq \frac{1}{(m\sqrt{M})^{2^{m-1}-1}}.$$
    
\end{lem}
\begin{proof}[Proof of Proposition \ref{Pro:MomentsCalculations}]
Let  $e(z):=\exp(2\pi i z)$. By writing 

\noindent $\cos\left(2\pi\alpha_0 \sqrt{mt}+\beta_0\right)= \left(e^{i\beta_0}e\left(\alpha_0\sqrt{mt}\right)+e^{-i\beta_0}e\left(-\alpha_0\sqrt{mt}\right)\right)/2$, we note that it is sufficient to show that 
\begin{align*}
&\frac1T\int_{T}^{2T}\left(\sum_{q\leq M} b_{q} e(\alpha_0\sqrt{qt})\right)^k\overline{\left(\sum_{m\leq M} b_{m} e(\alpha_0\sqrt{mt})\right)}^{\ell}dt\\
& = \ex\left(\left(\sum_{n_1r_1^2\leq M} b_{n_1r_1^2} e^{2\pi i r_1 \X(n_1)}\right)^k\overline{\left(\sum_{n_2r_2^2\leq M} b_{n_2r_2^2} e^{2\pi i r_2 \X(n_2)}\right)}^{\ell}\right)+ O\left(T^{-1/4}\right),
\end{align*}
for all integers $ k, \ell\geq 0 $ such that $k+\ell\leq h,$ where $b_m=a_me^{i\beta_0}.$ 
Expanding the moment we obtain 
\begin{equation}\label{Eq:ExpandMoments}
\begin{aligned}
    &\frac1T\int_{T}^{2T}\left(\sum_{q\leq M} b_{q} e(\alpha_0\sqrt{qt})\right)^k\overline{\left(\sum_{m\leq M} b_{m} e(\alpha_0\sqrt{mt})\right)}^{\ell}dt\\
    & = \sum_{\substack{q_1, \dots, q_k\leq M\\ m_1, \dots, m_{\ell}\leq M}} \prod_{i=1}^k b_{q_i}\prod_{j=1}^{\ell}\overline{b_{m_j}} \frac1T\int_T^{2T} e\left(\alpha_0\left(\sum_{i=1}^k\sqrt{q_i}-\sum_{j=1}^{\ell}\sqrt{m_j}\right)\sqrt{t}\right) dt.
\end{aligned}
\end{equation}
By writing $q_i=d_ie_i^2$ and $m_j=f_jg_j^2$, where $d_i, f_j$ are square-free, and using Besicovitch's Theorem \cite{Be} which states that the sequence $\{\sqrt{n}\}_{n\geq 1 \textup{ is square-free}}$ is linearly independent over $\mathbb{Q}$, together with the fact that the $\X(n)$ are independent, we deduce that  the contribution of the diagonal terms  $\sum_{i=1}^k\sqrt{q_i}-\sum_{j=1}^{\ell}\sqrt{m_j}=0$ equals%\footnote{This follows simply from Besicovitch's Theorem \cite{Be}.}
$$\ex\left(\left(\sum_{n_1r_1^2\leq M} b_{n_1r_1^2} e^{2\pi i r_1 \X(n_1)}\right)^k\overline{\left(\sum_{n_2r_2^2\leq M} b_{n_2r_2^2} e^{2\pi i r_2 \X(n_2)}\right)}^{\ell}\right).$$
Therefore, in order to complete the proof of the lemma, it remains to bound the contribution of the off-diagonal terms $\sum_{i=1}^k\sqrt{q_i}-\sum_{j=1}^{\ell}\sqrt{m_j}\neq 0$. By Lemma \ref{Lem:HughesRudnick} for each such term with $q_1, \dots, q_k,m_1, \dots, m_{\ell}\leq M$, we have 
$$\left|\sum_{i=1}^k\sqrt{q_i}-\sum_{j=1}^{\ell}\sqrt{m_j}\right|\geq \frac{1}{((k+\ell)\sqrt{M})^{2^{k+\ell-1}-1}}. $$
We now observe that if $\eta$ is any non-zero real number, then a simple integration by parts gives 
\begin{align*}\int_T^{2T} e(\eta\sqrt{t})dt&=
\int_T^{2T} (e(\eta\sqrt{t}))'\frac{\sqrt{t}}{\pi i \eta}dt\\
& =\left[ e(\eta\sqrt{t})\frac{\sqrt{t}}{\pi i \eta}\right]_T^{2T}- \int_{T}^{2T} e(\eta\sqrt{t})\frac{1}{2\pi i \eta \sqrt{t}}dt
\ll \frac{\sqrt{T}}{|\eta|}.
\end{align*}
  Thus we deduce that the contribution of the off-diagonal terms in \eqref{Eq:ExpandMoments} is 
  $$ \ll \frac{1}{\sqrt{T}}(c_{14}M)^{k+\ell}\left((k+\ell)\sqrt{M}\right)^{2^{k+\ell-1}-1}\ll T^{-1/4},$$
  for some positive constant $c_{14}$, by our assumptions on $a_m$, $k, \ell$ and $M$.
\end{proof}

%%%%%%%%%%%%%%%%%%%%%%%%%%%%%%%%%%%%%%%%%%%%%%%%%%%%%%%%%%%%%%%%%%%%%%%%%%%%%%%%%%%%%%
\section{The Laplace transform of $F(t)$}
In this section we establish the following key result, which shows that the Laplace transform of $F(t)$ (over a certain set of full measure) is very close to that of the probabilistic random model $F_{\X}$, in a certain uniform range of the parameters. 
\begin{thm}\label{Thm:Laplace}
There exists a positive constant $c_0$ and a set $\A\in [T, 2T]$ verifying

\noindent $\M([T, 2T]\setminus \A)\ll T(\log T)^{-10}$, such that for all complex numbers $\lambda$ with

\noindent $|\lambda|\leq c_0(\log_2 T)^{3/4}(\log_3 T)^{-9/4}$ we have 
$$\frac1T\int_{\A}\exp(\lambda F(t))dt = \ex\left(e^{\lambda F_{\X}}\right)+ O\left(\exp\left(-\frac{\log_2 T}{\log_3 T}\right)\right). $$
\end{thm}
\begin{rem}
Removing an exceptional set of small measure from the Laplace transform of $F(t)$ is necessary in order to prove Theorem \ref{Thm:Laplace}, since we do not have any control over the large values of $F(t)$. Indeed, by Theorem 2 of \cite{HB} only the first $9$ moments of $F(t)$ are known to exist.
\end{rem}
%HERE add remark on the set $\mathcal{A}$ that we basically remove the tail large and $F(t)$ large. 

%1. $F(t)$ large make laplace explode and

%2. Tail large make moments impossible to compute without a good truncations.
To prove Theorem \ref{Thm:Laplace} we record the following lemma which follows from the work of Heath-Brown \cite{HB}.
\begin{lem}\label{Lem:HeathBrown}
Let $T$ be large. For a positive integer $N$ we define  
$$ F_N(t):=\frac{1}{\pi \sqrt{2}}\sum_{n\leq N} \frac{\mu(n)^2}{n^{3/4}}\sum_{r\leq N}\frac{d(nr^2)}{r^{3/2}}\cos\left(4\pi r \sqrt{nt}-\frac{\pi}{4}\right).$$ If $ N\leq T^{1/4}$ then for any fixed $\varepsilon>0$ we have
$$ \frac{1}{T}\int_T^{2T}|F(t)-F_N(t)|dt\ll_{\varepsilon} \frac{1}{N^{1/4-\varepsilon}}.$$
\end{lem}
\begin{proof}
We let
$$ G_N(t):=\frac{1}{\pi \sqrt{2}}\sum_{n\leq N} \frac{\mu(n)^2}{n^{3/4}}\sum_{r=1}^{\infty}\frac{d(nr^2)}{r^{3/2}}\cos\left(4\pi r \sqrt{nt}-\frac{\pi}{4}\right).$$
Then, it follows from Equation (5.2) of \cite{HB} that for $U\geq N^2$ we have
\begin{equation}\label{Eq:HeathBrown}
    \int_{U}^{2U} \left|F(u^2)-G_N(u^2)\right|^2 du\ll_{\varepsilon}\frac{U}{N^{1/2-\varepsilon}}.
\end{equation}
Now, making the change of variables $t=u^2$ and using the Cauchy-Schwarz inequality we obtain 
\begin{equation}\label{Eq:HeathBrown2}
\begin{aligned}
\int_T^{2T}\left|F(t)-G_N(t)\right|dt & = 2\int_{\sqrt{T}}^{\sqrt{2T}}\left|F(u^2)-G_N(u^2)\right|u \ du \\
&\leq 2 \left(\int_{\sqrt{T}}^{\sqrt{2T}}\left|F(u^2)-G_N(u^2)\right|^2 du\right)^{1/2}\left(\int_{\sqrt{T}}^{\sqrt{2T}}u^2 \ du\right)^{1/2}\\
&\ll_{\varepsilon} \frac{T}{N^{1/4-\varepsilon}}. \\
\end{aligned}
\end{equation}
by \eqref{Eq:HeathBrown}. Finally, using the bound $d(n)\ll_{\varepsilon}n^{\varepsilon/3}$ we derive
\begin{align*}\sup_{t\in[T,2T]}\left|G_N(t)-F_N(t)\right| 
&\ll \sum_{n\leq N}\frac{1}{n^{3/4}}\sum_{r>N}\frac{d(nr^2)}{r^{3/2}}\\
&
\ll_{\varepsilon}\sum_{n\leq N} \frac{1}{n^{3/4-\varepsilon/3}}\sum_{r>N}\frac{1}{r^{3/2-2\varepsilon/3}}\ll_{\varepsilon}\frac{1}{N^{1/4-\varepsilon}}.
\end{align*}
Combining this bound with \eqref{Eq:HeathBrown2} completes the proof.
\end{proof}

\begin{proof}[Proof of Theorem \ref{Thm:Laplace}]
Let $N=\exp\left(\sqrt{\log T}\right)$ and $\A_1$ be the set of points $t\in [T, 2T]$ such that 
\begin{equation}\label{Eq:ConditionA1}
    \left|F(t)-F_N(t)\right|\leq N^{-1/10}.
\end{equation}
Then it follows from Lemma \ref{Lem:HeathBrown} that 
\begin{equation}\label{Eq:MeasureComplement2}
\M([T,2T]\setminus \A_1)\leq N^{1/10}\int_T^{2T}|F(t)-F_N(t)|dt\ll \frac{T}{N^{1/10}}. 
\end{equation}
We now let  $K:=\lfloor (\log\log T)/8\rfloor$  and  define $\A_2$ to be the set of points $t\in [T, 2T]$ such that 
\begin{equation}\label{Eq:ConditionA2}\left|F_N(t)\right|\leq V,
    \end{equation}
    where $V=C_3K^{1/4}(\log K)^{5/4}$ for some suitably large constant $C_3$. We also put
    $$ F_{N, \X}:=\frac{1}{\pi \sqrt{2}}\sum_{n\leq N} \frac{\mu(n)^2}{n^{3/4}}\sum_{r\leq N}\frac{d(nr^2)}{r^{3/2}}\cos\left(2\pi r \X_n-\frac{\pi}{4}\right), $$
    where $\{\X_n\}_{n \ \textup{square-free}}$ is a sequence of independent random variables, uniformly distributed on $[0,1]$.
    Then it follows from Proposition \ref{Pro:MomentsCalculations} that for any integer $0\leq k\leq 2K$ we have 
    \begin{equation}\label{Eq:MomentsAsymptotic}
    \frac{1}{T}\int_T^{2T}F_N(t)^kdt=  \ex\left(F_{N, \X}^{k}  \right) + O\left(T^{-1/4}\right). 
    \end{equation}
Therefore, using this asymptotic formula together with Proposition \ref{Pro:BoundMomentsRandom} gives
\begin{equation}\label{Eq:MeasureComplement}
\begin{aligned}
    \M([T, 2T]\setminus \A_2)& \leq  V^{-2K}\int_T^{2T}\left|F_N(t)\right|^{2K}dt\\
    & \ll V^{-2K} T \cdot \ex\left(|F_{N, \X}|^{2K}\right) + V^{-2K}T^{3/4}\\
    & \ll  T\left(\frac{2c_5K^{1/4}(\log K)^{5/4}}{V}\right)^{2K}\ll \frac{T}{(\log T)^{10}},
    \end{aligned}
    \end{equation}
if $C_3$ is suitably large.
We now define $\A:=\A_1\cap \A_2$ and put $\A^c=[T, 2T]\setminus \A$. Then combining \eqref{Eq:MeasureComplement2} and \eqref{Eq:MeasureComplement} we deduce that 
\begin{equation}\label{Eq:MeasureComplementFinal}
    \M(\A^c)\ll \frac{T}{(\log T)^{10}}.
\end{equation} 
Now by \eqref{Eq:ConditionA1} we get
\begin{equation}\label{Eq:ApproxLaplace1}
\begin{aligned}
\frac1T\int_{\A}\exp(\lambda F(t))dt 
& = \frac1T\int_{\A}\exp\left(\lambda F_N(t)+ O\left(|\lambda| N^{-1/10}\right)\right)dt \\
& = \frac1T\int_{\A}\exp\left(\lambda F_N(t)\right)dt + O\left( N^{-1/20}\right),
\end{aligned} 
\end{equation}
since 
$$\frac1T\int_{\A}\exp\left(\re (\lambda) F_N(t)\right)dt\leq \exp(\re (\lambda) V)\ll  N^{1/50},$$ 
by \eqref{Eq:ConditionA2} and our assumptions on $\lambda$ and $V$.
We now estimate  the main term on the right hand side of \eqref{Eq:ApproxLaplace1}. Let $L:= \lfloor \log_2 T/(\log_3 T)\rfloor$. By Stirling's formula and our assumption on $\A$ we have 
\begin{equation}\label{Eq:ApproxLaplace3}
\frac1T\int_{\A}\exp\left(\lambda F_N(t)\right)dt=\sum_{k=0}^{2L} \frac{\lambda^k}{ k!}\frac1T\int_{\A}F_N(t)^k dt+ E_1,
\end{equation}
where
$$ E_1 \ll \sum_{k>2L}\frac{(|\lambda| V)^k}{ k!}\ll \sum_{k>2L}\left(\frac{3|\lambda| V}{k}\right)^k\ll \sum_{k>2L}\left(\frac{3|\lambda| V}{2L}\right)^k\ll e^{-L},$$
if $c_0$ is suitably small. Furthermore, by the Cauchy-Schwarz inequality together with Proposition
\ref{Pro:BoundMomentsRandom} and Equations \eqref{Eq:MomentsAsymptotic} and \eqref{Eq:MeasureComplementFinal}, we deduce that for $k\leq 2L$ we have
\begin{align*}
    \left|\int_{\A^c} F_N(t)^k dt\right| &\leq \M(\A^c)^{1/2}\left(\int_T^{2T} |F_N(t)|^{2k}dt\right)^{1/2}\\
    & \ll \frac{\sqrt{T}}{(\log T)^{5}}\left(T\cdot \ex\left(|F_{N, \X}|^{2k}\right)+T^{3/4}\right)^{1/2}\\
    & \ll \frac{T}{(\log T)^{5}} \left(2c_5k^{1/4}(\log  2k)^{5/4}\right)^{k} \ll \frac{T}{(\log T)^{4}}.
\end{align*}
Inserting this estimate in \eqref{Eq:ApproxLaplace3} and using the asymptotic formula \eqref{Eq:MomentsAsymptotic} we get
\begin{equation}\label{Eq:ApproxLaplace4}
\frac1T\int_{\A}\exp(\lambda F_N(t))dt= \sum_{k=0}^{2L}\frac{\lambda^k}{k!} \ex\left(F_{N, \X}^k\right) + E_2
\end{equation}
where $$E_2\ll e^{-L}+ \frac{1}{(\log T)^{4}}\sum_{k=0}^{2L} \frac{|\lambda|^k}{k!}\ll  e^{-L}+\frac{e^{|\lambda|}}{(\log T)^{4}}\ll e^{-L}. $$ 
Combining this asymptotic formula with \eqref{Eq:ApproxLaplace1}
 we derive 
\begin{equation}\label{Eq:ApproxLaplace2}
\frac1T\int_{\A}\exp(\lambda F(t))dt= \sum_{k=0}^{2L}\frac{\lambda^k}{k!} \ex\left(F_{N, \X}^k\right) + O\left(e^{-L}\right).
\end{equation}
%where 
%$$ E_1 \ll e^{-L}+ N^{-1/1000}\sum_{k=0}^{2L} \frac{|\lambda|^k}{k!}(ck^{1/2})^k  + \frac{1}{(\log T)^{C/4}}\sum_{k=0}^{2L} \frac{|\lambda|^k}{k!} \ll e^{-L}+\frac{e^{|\lambda|}}{(\log T)^{C/4}}\ll e^{-L}.$$
To complete the proof we need to show that the main term on the right hand side of \eqref{Eq:ApproxLaplace2} is approximately $\ex\left(e^{\lambda F_{\X}}\right).$ To this end we define the event $\B$ by 
$$ \left|F_{\X}- F_{N, \X}\right|\leq N^{-1/10}.$$
    Then, it follows from Proposition \ref{Pro:BoundMomentsRandom} that 
    \begin{equation}\label{Eq:ProbabilityTailRandom} 
    \begin{aligned}
        \pr(\B^c) &\leq N^{1/5} \ex\left(\left|F_{\X}- F_{N, \X}\right|^{2}\right)\\
        &\ll N^{1/5}\ex\left(\left|\sum_{n> N} \frac{\mu(n)^2}{n^{3/4}}\sum_{r=1}^{\infty}\frac{d(nr^2)}{r^{3/2}}\cos\left(2\pi r \X_n-\frac{\pi}{4}\right)\right|^{2}\right)+ N^{-1/4}\\
        &\ll N^{-1/4},
        \end{aligned}
  \end{equation}
  since 
  $$\sum_{n\leq  N} \frac{\mu(n)^2}{n^{3/4}}\sum_{r>N}^{\infty}\frac{d(nr^2)}{r^{3/2}}\ll_{\ep}N^{-1/4+\ep} .$$
    Let $\mathbf{1}_{\mathcal{C}}$ denote the indicator function of an event $\mathcal{C}$, and let $k\leq 2L$ be a positive integer. Then we observe that 
    \begin{align*}
\ex\left(F_{\X}^k\right)&=\ex\left(\mathbf{1}_{\B} \cdot F_{\X}^k\right)+ \ex\left(\mathbf{1}_{\B^c}\cdot F_{\X}^k\right)\\
&= \ex\left(\mathbf{1}_{\B} \cdot \left(F_{N, \X}+O(N^{-1/10}\right)^k\right)+ \ex\left(\mathbf{1}_{\B^c}\cdot F_{\X}^k\right)\\
&=\ex\left(\left(F_{N, \X}+O(N^{-1/10}\right)^k\right)+ O\left(\left|\ex\left(\mathbf{1}_{\B^c} \cdot \left(F_{N, \X}+O(N^{-1/10}\right)^k\right)\right|+ \left|\ex\left(\mathbf{1}_{\B^c}\cdot F_{\X}^k\right)\right|\right).\end{align*}
 First we handle the main term. By the Binomial Theorem together with Proposition \ref{Pro:BoundMomentsRandom}  this equals
 $$\ex\left(F_{N, \X}^k\right) +E_3,$$
 where 
 $$ E_3 \ll \sum_{j=1}^k \binom{k}{j}(c_{15}N)^{-j/10}\cdot\ex\left(|F_{N, \X}|^{k-j}\right)\ll N^{-1/10}\left(c_5k^{1/4}(\log  2k)^{5/4}\right)^{k}\ll N^{-1/20},$$
 for some positive constant $c_{15}.$
 Next, by the Cauchy-Schwarz inequality, Proposition \ref{Pro:BoundMomentsRandom} and the estimate \eqref{Eq:ProbabilityTailRandom} we deduce that 
\begin{align*}
&\left|\ex\left(\mathbf{1}_{\B^c} \cdot \left(F_{N, \X}+O(N^{-1/10}\right)^k\right)\right|+ \left|\ex\left(\mathbf{1}_{\B^c}\cdot F_{\X}^k\right)\right|\\
& \leq \pr(\B^c)^{1/2} \left(\ex\left(\left|F_{N, \X}+O(N^{-1/10})\right|^{2k}\right)^{1/2}+ \left|\ex\left(|F_{\X}|^{2k}\right)\right|^{1/2}\right)\\
& \ll N^{-1/8} \left(2c_5k^{1/4}(\log  2k)^{5/4}\right)^{k} \ll N^{-1/10}.
\end{align*}
%since 
%$\ex\left(|\F|^{2k}\right)\ll \left(c_4k^{1/4}(\log 2k)^{5/4}\right)^{2k}$ by \eqref{Eq:BoundMomentsTail}.
Collecting the above estimates, we obtain 
$$ \ex\left(F_{N,\X}^k\right)= \ex\left(F_{\X}^k\right) +O\left(N^{-1/20}\right).$$
We now insert this asymptotic formula in \eqref{Eq:ApproxLaplace2} to get
\begin{align*}\frac1T\int_{\A}\exp(\lambda F(t))dt&= \sum_{k=0}^{2L}\frac{\lambda^k}{k!} \ex\left(F_{ \X}^k\right) + O\left(N^{-1/20}e^{|\lambda|}+e^{-L}\right)\\
&= \ex\left(e^{\lambda F_{\X}}\right)+E_4,
\end{align*}
where 
$$E_4\ll e^{-L}+ \sum_{k>2L}\frac{|\lambda|^k}{k!}\ex\left(|F_{\X}|^{k}\right)\ll e^{-L} + \sum_{k>2L}\left(\frac{3c_5|\lambda| (\log 2k)^{5/4}}{k^{3/4}} \right)^k,$$
by Stirling's formula and Proposition \ref{Pro:BoundMomentsRandom}. Finally, our assumption on $\lambda$ insures that 
$$ \sum_{k>2L}\left(\frac{3c_5|\lambda| (\log 2k)^{5/4}}{k^{3/4}} \right)^k\ll \sum_{k>2L}\left(\frac{3c_5|\lambda| (\log 4L)^{5/4}}{L^{3/4}} \right)^k \ll e^{-L}, $$
 completing the proof.
\end{proof}

%%%%%%%%%%%%%%%%%%%%%%%%%%%%%%%%%%%%%%%%%%%%%%%%%%%%%%%%%%%%%%%%%%%%%%%%%%%%%%%%%%%%%%%%%%%%%
\section{The discrepancy and large deviations of the distribution: Proofs of Theorems \ref{Thm:Discrepancy} and \ref{Thm:LargeDeviations}}
%%%%%%%%%%%%%%%%%%%%%%%%%%%%%%%%%%%%%%%%%%%%%%%%%%%%%%%%%%%%%%%%%%%%%%%%%%%%%%%%%%%%%%%%%%%%%%%%
\subsection{The discrepancy: Proof of Theorem \ref{Thm:Discrepancy}}
For a real number $\alpha$ we define 
$$\ph_{F,T}(\alpha):= 
\frac{1}{T} \int_T^{2T} e^{i\alpha F(t)}dt, \text{ and } \ph_{F_{\X}}(\alpha):= \ex\left(e^{i\alpha F_{\X}}\right).$$
Then it follows from Theorem \ref{Thm:Laplace} that uniformly for $\alpha$ in the range 

\noindent $|\alpha|\leq c_0(\log_2 T)^{3/4}(\log_3 T)^{-9/4}$ we have 
\begin{equation}\label{Eq:ApproxFourier}
    \ph_{F,T}(\alpha)= \frac{1}{T} \int_{\A} e^{i\alpha F(t)}dt +O\left(\frac{\M(\A^c)}{T}\right)= \ph_{F_{\X}}(\alpha)+ O\left(\exp\left(-\frac{\log_2 T}{\log_3 T}\right)\right).
\end{equation}
Let $R:=c_0(\log_2 T)^{3/4}(\log_3 T)^{-9/4}$. By the Berry-Esseen inequality we have 
\begin{equation}\label{Eq:BerryEsseen1}
    \sup_{u\in \mathbb{R}}\left|\pr_T(F(t)\leq u)-\pr(F_{\X}\leq u)\right|\ll \frac1R+ \int_{-R}^R\left|\frac{\ph_{F,T}(\alpha)-\ph_{F_{\X}}(\alpha)}{\alpha}\right|d\alpha.
\end{equation}
In the range $1/\log T\leq |\alpha|\leq R$ we use \eqref{Eq:ApproxFourier} which gives 
\begin{equation}\label{Eq:BerryEsseen2}
%\begin{aligned}
\int_{1/\log T \leq |\alpha|\leq R} \left|\frac{\ph_{F,T}(\alpha)-\ph_{F_{\X}}(\alpha)}{\alpha}\right|d\alpha 
%&\ll \exp\left(-\frac{\log_2 T}{\log_3 T}\right) \int_{1/\log T \leq |\alpha|\leq R} \frac{d\alpha}{\alpha}\\
\ll \exp\left(-\frac{\log_2 T}{2\log_3 T}\right).
%\end{aligned}
\end{equation}
We now handle the remaining range $0\leq |\alpha| \leq 1/\log T$. Using the inequality $|e^{iv}-1|\ll |v|$, valid for all real numbers $v$, we obtain 
$$ \ph_{F_{\X}}(\alpha)= 1+ O\left(|\alpha|\ex(|F_{\X}|)\right)= 1+ O\left(|\alpha|\right),
$$ 
and similarly
$$ \ph_{F,T}(\alpha)= 1+ O\left(|\alpha|\frac{1}{T}\int_{T}^{2T} |F(t)|dt\right)=1+ O\left(|\alpha|\right),$$
where the last estimate follows from  Theorem 2 of \cite{HB}.
%Moreover, by the Cauchy-Schwarz inequality and Heath-Brown's result \cite{HB} that the second moment of $F(t)$ is bounded we get
%$$\frac{1}{T}\int_{T}^{2T} |F(t)|dt \leq \left(\frac{1}{T}\int_{T}^{2T} |F(t)|^2dt\right)^{1/2} \ll 1.$$
%This implies $$\ph_{F,T}(\alpha)= 1+ O(|\alpha|).$$ 
Therefore we deduce that 
$$\int_{-1/\log T}^{1/\log T}\left|\frac{\ph_{F,T}(\alpha)-\ph_{F_{\X}}(\alpha)}{\alpha}\right|d\alpha \ll \int_{-1/\log T}^{1/\log T} 1 d\alpha\ll \frac{1}{\log T}.$$
Combining this bound with \eqref{Eq:BerryEsseen1} and \eqref{Eq:BerryEsseen2} completes the proof.
\subsection{Large deviations: Proof of Theorem \ref{Thm:LargeDeviations}}
In \cite{Lau1}, Lau extracts the bounds \eqref{Eq:LAUBounds} from Lemma \ref{Lem:Lau} using a Tauberian argument. However, this method does not give uniform results. To get uniformity and deduce Theorem \ref{Thm:LargeDeviations} from Theorem \ref{Thm:Laplace}, we use a simpler  and more direct probabilistic approach. %which relies on ideas of Montgomery and Odlyzko \cite{MoOd}. 
Indeed, in order to prove the upper bound of Theorem \ref{Thm:LargeDeviations} we use Chernov's inequality, while for the proof of the lower bound we shall use the classical Paley-Zygmund inequality. 

%following elementary result (see for example \cite[Lemma 1]{MoOd}). \begin{lem}[Lemma 1 of \cite{MoOd}]\label{Lem:MontgomeryOdlyzko}
%Let $Z$ be a random variable on a probability space $(\Omega, \mathcal{F}, \pr)$ such that $Z\in L^2(\Omega)$. For any real number $0<a<1$ we have 
%$$ \pr(Z\geq a \ex(Z)) \geq (1-a)^2 \frac{\ex(Z)^2}{\ex(Z^2)}.$$
    
%\end{lem}

\begin{proof}[Proof of Theorem \ref{Thm:LargeDeviations}]
We only prove the result for $\pr_T(F(t)>V)$, since the proof for $\pr_T(F(t)<-V)$ follows along the same lines by considering $-F(t)$ instead of $F(t)$. We start by proving the upper bound.
Let $\A$ be the set in Theorem \ref{Thm:Laplace} and $0<\lambda\leq c_0(\log_2T)^{3/4}(\log_3 T)^{-9/4}$ be a parameter to be chosen. By Theorem \ref{Thm:Laplace} we have 
\begin{equation}\label{Eq:Chernov}
\begin{aligned}
\pr_T(F(t)>V)
&\leq \frac1T\M(\A^c) + \frac1T\int_{\A}\exp\left(\lambda(F(t)-V)\right)dt\\
& \ll \frac{1}{(\log T)^{10}}+ e^{-\lambda V} \ex\left(e^{\lambda F_{\X}}\right).
\end{aligned}
\end{equation}
 Using Lemma \ref{Lem:Lau} and choosing $\lambda$ such that $V= 2c_2 \lambda^{1/3}(\log \lambda)^{2^{4/3}-1}$ we obtain 
$$\pr_T(F(t)>V) \ll \frac{1}{(\log T)^{10}} + e^{-\lambda V/2} \ll \exp\left(-c_{16} V^4(\log V)^{-3(2^{4/3}-1)}\right),$$
for some positive constant $c_{16}.$   

We now prove the lower bound. Let $\lambda_{\textup{max}}=c_0(\log_2 T)^{3/4}(\log_3 T)^{-9/4}$, where $c_0$ is the constant in Theorem \ref{Thm:Laplace}.  Let $\lambda$ be such that 
\begin{equation}\label{Eq:ChoiceLambdaLower}
\frac{1}{\M(\A)}\int_{\A} e^{\lambda F(t)}dt= 2e^{\lambda V}.
\end{equation}
Such a $\lambda$ exists and verifies $0<\lambda<\lambda_{\textup{max}}$, since both sides above are continuous functions of $\lambda$, the left hand side is smaller than the right when $\lambda=0$, and the reverse is true when $\lambda=\lambda_{\textup{max}}$, since by Theorem \ref{Thm:Laplace} and Lemma \ref{Lem:Lau} we have 
\begin{align*}
   \frac{1}{\M(\A)}\int_{\A} e^{\lambda_{\textup{max}} F(t)}dt&= (1+o(1))\ex\left(e^{\lambda_{\textup{max}} F_{\X}}\right)+o(1)\\
   &\geq \exp\left(\frac{c_1}{2}\lambda_{\textup{max}}^{4/3}\left(\log\lambda_{\textup{max}}\right)^{2^{4/3}-1}\right)> 2 e^{\lambda_{\textup{max}}V}
\end{align*}
by our assumption on $V$, if $b_2$ is suitably small.
Therefore, combining Theorem \ref{Thm:Laplace} with Lemma \ref{Lem:Lau} and Equation \eqref{Eq:ChoiceLambdaLower} we deduce that 
$$ \lambda V= \log\left(\ex\left(e^{\lambda F_{\X}}\right)\right)+O(1)\asymp \lambda^{4/3}(\log \lambda)^{2^{4/3}-1},$$
which  implies 
\begin{equation}\label{Eq:LambdaV}
    \lambda \asymp V^3 (\log V)^{-3(2^{4/3}-1)}.
\end{equation} 
On the other hand, by %Lemma \ref{Lem:MontgomeryOdlyzko}
the Paley-Zygmund inequality and our assumption \eqref{Eq:ChoiceLambdaLower} we obtain 
\begin{align*}
\pr_T(F(t)>V)&= \frac{1}{\M(\A)}\M\{t\in \A : e^{\lambda F(t)}>e^{\lambda V}\} + O\left(\frac{1}{(\log T)^{10}}\right)\\
&\geq \frac{1}{4}\frac{\displaystyle{\left( \frac{1}{\M(\A)}\int_{A} e^{\lambda F(t)}dt\right)^2}}{\displaystyle{\frac{1}{\M(\A)}\int_{A} e^{2\lambda F(t)}dt}}+ O\left(\frac{1}{(\log T)^{10}}\right).
\end{align*} 
We now use our assumption \eqref{Eq:ChoiceLambdaLower} together with Theorem \ref{Thm:Laplace} and Lemma \ref{Lem:Lau} to obtain 
\begin{align*}
\pr_T(F(t)>V)\gg \frac{e^{2\lambda V}}{\ex\left(e^{2\lambda F_{\X}}\right)}+ \frac{1}{(\log T)^{10}}\gg \exp\left(-c_{17}\lambda^{4/3}(\log \lambda)^{2^{4/3}-1}\right),
\end{align*} 
for some positive constant $c_{17}$ by our assumption on $\lambda$ and using the trivial bound $e^{\lambda V}\geq 1$. Combining this bound with \eqref{Eq:LambdaV} completes the proof. 
\end{proof}

%%%%%%%%%%%%%%%%%%%%%%%%%%%%%%%%%%%%%%%%%%%%%%%%%%%%%%%%%%%%%%%%%%%%
\section{Analogous results for $P(t)$ and $E(T)$: Proofs of Theorems \ref{Thm:DiscrepancyCircle}, \ref{Thm:LargeDeviationsCircle}, \ref{Thm:DiscrepancyE(T)} and \ref{Thm:LargeDeviationsE(T)}}
\subsection{The error term is the second moment of $\zeta(1/2+it)$}
We start by handling the case of $E(T)$ first since this case is the closest to $\Delta(t).$ Using Atkinson's formula for $E(T)$,  Heath-Brown (see Section 6 of \cite{HB}) proved that  uniformly for $N\leq U^{1/8}$ we have the following similar bound to \eqref{Eq:HeathBrown}
\begin{equation}\label{Eq:HeathBrownE(t)}
\frac1U\int_U^{2U} |u^{-1/2}E(u^2)-E_N(u^2)|^2du\ll\frac{1}{N^{1/8}},
\end{equation}
where 
$$E_N(t):=\left(\frac{2}{\pi}\right)^{1/4}\sum_{n\leq N}\frac{\mu(n)^2}{n^{3/4}}\sum_{nr^2\leq N^{4}}(-1)^{nr^2}\frac{d(nr^2)}{r^{3/2}}\cos\left(2
\pi \alpha_0 r\sqrt{nt} -\frac{\pi}{4}\right),$$
and $\alpha_0=\sqrt{2/\pi}.$
Therefore, following the exact same lines of the  proof of Theorem \ref{Thm:Laplace}, replacing \eqref{Eq:HeathBrown} by \eqref{Eq:HeathBrownE(t)}, and using Proposition \ref{Pro:MomentsCalculations} with $\alpha_0=\sqrt{2/\pi}$ and $\beta_0=-\pi/4$, as well as  Proposition \ref{Pro:BoundMomentsRandom}
with $a_{nr^2}=(-1)^{nr^2}$ we deduce the following result:
\begin{thm}\label{Thm:LaplaceE(T)}
There exists a positive constant $c_{18}$ and a set $\A\in [T, 2T]$ verifying

\noindent $\M([T, 2T]\setminus \A)\ll T(\log T)^{-10}$, such that for all complex numbers $\lambda$ with

\noindent $|\lambda|\leq c_{18}(\log_2 T)^{3/4}(\log_3 T)^{-9/4}$ we have 
$$\frac1T\int_{\A}\exp(\lambda t^{-1/4}E(t))dt = \ex\left(e^{\lambda E_{\X}}\right)+ O\left(\exp\left(-\frac{\log_2 T}{\log_3 T}\right)\right). $$
\end{thm}
 From this result one can deduce Theorems \ref{Thm:DiscrepancyE(T)} and \ref{Thm:LargeDeviationsE(T)} along the exact same lines of the proofs of Theorems \ref{Thm:Discrepancy} and \ref{Thm:LargeDeviations}. %, which Lau established for the probabilistic random model $E_{\X}$ (see Theorem 2 of \cite{Lau1}). 

\subsection{The remainder term in the circle problem}
It follows from the work of Heath-Brown (see Sections 5 and 6 of \cite{HB}) 
that for
$U\geq N^2$ we have
\begin{equation}\label{Eq:HeathBrownP(t)}
    \frac1U\int_{U}^{2U} \left|u^{-1/2}P(u^2)-P_N(u^2)\right|^2 du\ll_{\varepsilon}\frac{1}{N^{1/2-\varepsilon}},
\end{equation}
where 
$$P_N(t):=-\frac{1}{\pi }\sum_{n\leq N} \frac{\mu(n)^2}{n^{3/4}}\sum_{q=1}^{\infty}\frac{r(nq^2)}{q^{3/2}}\cos\left(2\pi q \sqrt{nt}+\frac{\pi}{4}\right).$$
%together with the method of proof of Lemma \ref{Lem:HeathBrown} that for $ N\leq T^{1/4}$ and any fixed $\varepsilon>0$ we have
%\begin{equation}\label{Eq:HeathBrownP(t)}
%\frac1T\int_T^{2T} |t^{-1/4}P(t)-P_N(t)|dt\ll_{\ep}\frac{1}{T^{1/4}},
%\end{equation}
 Therefore, following the same lines of the  proof of Theorem \ref{Thm:Laplace}, replacing \eqref{Eq:HeathBrown} by \eqref{Eq:HeathBrownP(t)},  using Proposition \ref{Pro:MomentsCalculations} with $\alpha_0=1$ and $\beta_0=\pi/4$, and replacing \eqref{Eq:BoundMomentsTail} by \eqref{Eq:BoundMomentsTailCircle} in Proposition \ref{Pro:BoundMomentsRandom} we deduce the following result:
\begin{thm}\label{Thm:LaplaceCircle}
There exists a positive constant $c_{19}$ and a set $\A\in [T, 2T]$ verifying

\noindent $\M([T, 2T]\setminus \A)\ll T(\log T)^{-10}$, such that for all complex numbers $\lambda$ with

\noindent $|\lambda|\leq c_{19}(\log_2 T)^{3/4}(\log_3 T)^{-5/4}$ we have 
$$\frac1T\int_{\A}\exp(\lambda t^{-1/4}P(t))dt = \ex\left(e^{\lambda P_{\X}}\right)+ O\left(\exp\left(-\frac{\log_2 T}{\log_3 T}\right)\right). $$
\end{thm}
 From this result one extracts Theorems \ref{Thm:DiscrepancyCircle} and \ref{Thm:LargeDeviationsCircle} in the exact same manner, using part 2 of %the analogue of 
 Lemma \ref{Lem:Lau} in this case.

\end{document}